\documentclass{article}

\usepackage{amssymb}
\usepackage{latexsym}
\usepackage{amsthm}
\usepackage{enumerate}
\usepackage{epsfig}
\usepackage{color}
\usepackage{float}
\usepackage{amsmath}
\usepackage{tikz}
\usetikzlibrary{arrows}
\usepackage[nobysame, alphabetic]{amsrefs}

\def\C{\mathcal{C}}

\def\N{\mathbb{N}}

\def\T{\mathcal{T}}

\def\cross{\times}

\def\e{\epsilon}

\def\PMF{\mathcal{PMF}}

\def\CS{\mathcal{C}_{q}}

\def\MCG{\mathrm{MCG}}

\def\diam{\text{diam}}

\renewcommand{\ge}{\geqslant}
\renewcommand{\le}{\leqslant}

\newcommand{\norm}[1]{|#1|}

\newtheorem{theorem}{Theorem}[section]
\newtheorem{lemma}[theorem]{Lemma}
\newtheorem{corollary}[theorem]{Corollary}
\newtheorem{proposition}[theorem]{Proposition}

\theoremstyle{definition}

\newtheorem{remark}[theorem]{Remark}

\newtheorem{definition}[theorem]{Definition}

\restylefloat{figure}

\begin{document}


\title{Quotients of the curve complex.}
\author{Joseph Maher, Hidetoshi Masai, Saul Schleimer}
\date{\today}

\maketitle

\begin{abstract}
We consider three kinds of quotients of the curve complex which are obtained by coning off uniformly quasi-convex subspaces:
symmetric curve sets, non-maximal train track sets, and compression body disc sets.
We show that the actions of the mapping class group on those quotients are strongly WPD, 
which implies that the actions are non-elementary and those quotients are of infinite diameter.

Subject code: 37E30, 20F65, 57M50.
\end{abstract}

\tableofcontents

\section{Introduction}

If $X$ is a Gromov hyperbolic space, and $Y =\{ Y_i \}$ is a
collection uniformly quasiconvex subsets of $X$, then the
electrification $X_Y$ is Gromov hyperbolic.  Furthermore, if a group
$G$ acts on $X$ by isometries, and $Y$ is $G$-equivariant, then $G$
also acts on the electrification $X_Y$ by isometries.

However, if the action of $G$ on $X$ has some useful property
(e.g. non-elementary, acylindrical, WPD), then it is not known whether
the corresponding action on $X_Y$ has this property.  In particular,
in general, $X_Y$ need not be infinite diameter.

We shall consider the action of the mapping class group of a finite
type orientable surface on the curve complex of the surface $\C(S)$.
Our results apply to the following families of uniformly quasiconvex
subsets.

The symmetric curve sets: If $q\colon S \to S'$ is an orbifold
quotient of $S$, then we shall denote by $C_q(S)$ the collection of
all $q$-equivariant curves in $S$.  The collection $\{ \C_q(S) \}$ as
$q$ runs over all orbifold quotients of $S$ is a mapping class group
invariant, uniformly quasiconvex family of subsets of $\C(S)$
\cite{RS}.  We shall denote the electrification by $\C_q(S)$.

Non-maximal train track sets: A train track $\tau$ on $S$ is
\emph{maximal} if every complementary region is a triangle or a
monogon containing a single puncture.  Given a train track $\tau$, we
shall write $\C(\tau)$ for all simple closed curves carried by $\tau$.
The set $\{ \C(\tau) \}$ as $\tau$ runs over all non-maximal train
tracks in $S$ is a mapping class group invariant, uniformly
quasiconvex family of subsets of $\C(S)$ \cite{masur-minsky}.  We
shall denote the electrification by $\C_\tau(S)$.

Compression body disc sets: Let $V$ be a compression body with
boundary $S$, and let $D(V)$ be the collection of isotopy classes of
essential simple closed curves in $S$ which bound discs in $V$.  The
set $\{ D(V) \}$, as $V$ runs over all compression bodies, forms a
mapping class group invariant, uniformly quasiconvex family of subsets
of $\C(S)$ \cite{masur-minsky}.  We shall write $\C_D(S)$ for the
electrification, which is quasi-isometric to the compression body
graph.

We shall show:

\begin{theorem}\label{thm.main}
Both $\C_q(S), \C_\tau(S)$ and $\C_D(S)$ are infinite diameter Gromov
hyperbolic spaces, and furthermore, the action of the mapping class
group on these spaces is strongly WPD.
\end{theorem}

See Sections \ref{section:co-symmetric}, \ref{section:non-maximal} and
\ref{section:compression body} for more detailed statements.  In the
case of $\C_\tau(S)$ this is a special case of a much more general
result of Hamenst\"adt, which has not yet appeared in print.  In the
case of $\C_D(S)$, the compression body graph was previously shown to
have infinite diameter in \cite{maher-schleimer}.

For simplicity in this introduction {\em we only consider random walks
  on the mapping class group $\MCG(S)$ whose transition probability
  has the support generating the whole $\MCG(S)$}.  In \cite{MT},
Maher-Tiozzo proved that in our situation, any random walk gives rise
to a loxodoromic element with asymptotic probability one.  Hence in
particular, the following two previously known results are now
immediate corollaries of this results.  Furthermore, this gives
largely independent proofs of these results.

\begin{theorem}\cite{Mas}
Let $S$ be a surface of finite type of gnerus $g$ with $n$ punctures,
and with $2g + n > 4$. Then a random walk on the mapping class group
of $S$ gives rise to an asymmetric pseudo-Anosov element with
asymptotic probability one.
\end{theorem}

We remark that here we say a pseudo-Anosov map on $S$ is
\emph{asymmetric} if it is not a lift of a pseudo-Anosov map on a
surface $S'$ which admits an orbifold cover $p \colon S \to S'$.  This
cover need not be regular, so this is a more general condition than
requiring some power of the pseudo-Anosov element commute with a
finite order element.  The condition $2g + n > 4$ includes all
surfaces which support a pseudo-Anosov map, but excludes the four
exceptional surfaces $S_{0, 4}, S_{1, 1}, S_{1, 2}$ and $S_{2, 0}$.
The result does not hold for these surfaces, as all pseudo-Anosov maps
on these surfaces are symmetric with respect to the hyperelliptic
involutions on the surfaces.

We say a pseudo-Anosov map has \emph{generic singularities} if its
invariant foliations have $1$-prong singularities at each puncture,
and all other singularities are trivalent.

\begin{theorem}\cite{gadre-maher}\label{theorem:gadre-maher}
Let $S$ be a surface of finite type of genus $g$ with $n$ punctures,
and with $2g + n > 3$.  Then a random walk on the mapping class group
of $S$ gives rise to a pseudo-Anosov element whose invariant
foliations have generic singularities with asymptotic probability one.
\end{theorem}

The condition $2g+n > 3$ includes all finite type surfaces which
support a pseudo-Anosov map with the exception of $S_{1, 1}$ for which
the result does not hold, see Remark \ref{remark:stratum}.

Furthermore, our result also has an application to the Poisson
boundary.  In \cite{KM}, it is proved that the space of projective
measured foliations $\PMF$ has a unique harmonic measure $\nu$, and
the measure space $(\PMF,\nu)$ is a Poisson boundary of the random
walk on $\MCG(S)$.  Maher-Tiozzo \cite{MT} proved that if a group acts
acylindrically on a Gromov hyperbolic space, then the Gromov boundary
equipped with the hitting (harmonic) measure is a Poisson boundary.
By the work of Osin \cite{Osin}, any WPD action has a quotient with
acylindrical action.  Combining these with a characterization of the
boundary of quotient (Proposition \ref{prop.boundary} below) Theorem
\ref{thm.main} gives the following.

\begin{corollary}
Let $S$ be a surface of finite type of genus $g$ with $n$ punctures,
and with $2g + n > 4$.  Let $\nu$ be the harmonic measure arising from
a random walk on $\MCG(S)$.  Then the harmonic measure $\nu$ on
$\PMF(S)$ is concentrated on asymmetric and maximal foliations.
\end{corollary}

\section{Quotients of hyperbolic graphs}

Given a metric space $(X, d)$, and a collection of subsets
$Y = \{ Y_i \}_{i \in I} $, the \emph{electrification} of $X$ with
respect to $Y$ is the metric space $X_Y$ obtained by adding a new
vertex $y_i$ for each set $Y_i$, and coning off $Y_i$ by attaching
edges of length $\tfrac{1}{2}$ from each $y \in Y_i$ to $y_i$.  The
metric on $X_Y$ is the induced path metric, and so the image of each
set $Y_i$ in $X_Y$ has diameter at most one.

We say the family of sets $Y$ is \emph{$G$-equivariant} if for all
$g \in G$ and each $Y_i \in Y$ then $g Y_i$ is also in $Y$.


\begin{theorem} \label{prop.hyperbolic} Let $X$ be a Gromov
hyperbolic space and let $Y = \{ Y_i\}_{i \in I}$ be a collection
of uniformly quasi-convex subsets of $X$.  Then $X_Y$, the
electrification of $X$ with respect to $Y$, is Gromov hyperbolic.

Furthermore, if a group $G$ acts on $X$ by isometries, and $Y$ is
$G$-equivariant, then $G$ also acts on $X_Y$ by isometries.
%
\end{theorem}

The first statements are due to Bowditch
\cite{bowditch-rel-hyp}*{Proposition 7.12}, and Kapovich and Rafi
\cite{kapovich-rafi}*{Corollary 2.4}.  The second statement is
immediate.

\subsection{WPD action on quotients}

We recall the definition of weakly properly discontinuous action
defined by Bestvina-Fujiwara \cite{bestvina-fujiwara}.
\begin{definition}
Let $X$ be a connected metric graph and $G$ a group of isometries of
$X$.  An element $g\in G$ acts {\em loxodromically}, or is {\em
  loxodromic} if the map $\mathbb{Z}\rightarrow X$ given by
$n\mapsto g^{n}x$ is a quasi-isometric embedding for any $x\in X$.  A
loxodromic element $g\in G$ is said to act {\em weakly properly
  discontinuously(WPD)} or is {\em WPD} if for any $r>0$ and $x\in X$,
there exists $N\in\mathbb{N}$ such that
$$\sharp\left\{h\in G\mid d(x,hx)<r \text{ and } d(g^{N}x, hg^{N}x)<r \right\}<\infty.$$

We say that the action of $G$ on $X$ is {\em WPD} if
\begin{itemize}
\item $G$ is not virtually cyclic,
\item there is at least one loxodromic element in $G$ which is WPD.
\end{itemize}
\end{definition}

We say that $G$ is \emph{strongly WPD} if every loxodromic element is
WPD.

Note that if there is a loxodromic element acting on a graph, then the graph must have infinite diameter.
Moreover, Bestvina-Fujiwara observed the following.
\begin{proposition}[{\cite[Proposition 6]{bestvina-fujiwara}}]\label{prop.non-elementary}
Suppose the action of $G$ on a Gromov hyperbolic space $X$ is WPD, then
the action of $G$ on $X$ is non-elementary.
\end{proposition}

We now show that under certain conditions, if $g$ acts WPD on a Gromov
hyperbolic metric graph, then $g$ is also WPD on the quotient graph we
get by coning off.

\begin{theorem}\label{thm.WPD}
Let the group $G$ act by isometries on the hyperbolic metric space
$X$, and let $Y$ be a family of uniformly quasiconvex $G$-equivariant
subsets of $X$.

Let $g\in G$ act loxodromically on $X$ and $\alpha$ be a
quasi-geodesic axis of $g$.  We denote by $\pi_{\alpha}$ the nearest
projection map to $\alpha$ on $X$.  Suppose
\begin{enumerate}
\item[(1)] the action of $g$ on $X$ is WPD, and
\item[(2)] there is a constant $C>0$ such that for every element $Y_{i}\in Y$, 
$$\diam_{X}(\pi_{\alpha}(Y_{i}))<C.$$
\end{enumerate}
Then the action of $g$ on $X_Y$ is also WPD.
\end{theorem}

\begin{proof}
For notational simplicity let $Z:= X_{Y}$.  We first observe that
$\alpha$ has infinite diameter in $Z$.  Suppose, not then there is a
$D$ such that any pair of points $x$ and $y$ in $\alpha$ are distance
at most $D$ apart in $Z$.  Then there is a path consisting of $D$
segments in the $Y_i$ connecting $x$ to $y$.  By assumption (2), the
nearest point projection of each path to $\alpha$ in $X$ has diameter
at most $C$, the distance between $x$ and $y$ in $X$ is at most $CD$,
a contradiction.

Let $z$ be a point in $X$, and hence also a point in $Z$, and let
$B_{Z}(z,r)$ denote the ball of radius $r$ centered at $z$ in $Z$.  We
denote by $\widehat B_{Z}(z,r)$ the restriction of $B_{Z}(z,r)$ to
$X$.  Then by assumption (2), we see that
$\diam_{X} (\pi_{\alpha}(\widehat B(z,r)) < C r$.  Recall that in a
Gromov hyperbolic space, if the nearest projections to a
quasi-geodesic $\alpha$ of given two points $x,y$ are reasonably far
apart, then the geodesic connecting $x$ and $y$ must have a subarc
fellow traveling with $\alpha$.  Therefore for large enough $N$, there
is a subarc $\beta\subset\alpha$ and $D_{1}>0$ such that any geodesic
in $X$ connecting a point in $\widehat B_{Z}(z,r)$ and a point in
$\widehat B_{Z}(g^{N}z,r)$ must contain $\beta$ in its
$D_{1}$-neighbourhood.  Note that we may suppose $\beta$ as long as we
need by taking large enough $N$.
Let $R:=d_{X}(z,\alpha)$.  Suppose there is $h\in G$ such that
$d_{Z}(z,hz)<r$ and $d_{Z}(g^{N}z,hg^{N}z)<r$.  Then any geodesic
$\gamma$ in $X$ connecting $hz$ and $hg^{N}z$ contains $\beta$ in its
$D_{1}$-neighbourhood.  This implies that there exists a constant
$D_{2} > C r$ such that for any point $b\in\beta$ at least $R+D_{2}$
apart from the endpoints, we have $d_{X}(b, hb)\leq R+2D_{1}+D_{2}$.
Note that the constant $R+2D_{1}+D_{2}$ is independent of $N$.
Therefore by taking $N$ large enough, we may apply the fact that $g$
acts WPD on $X$ to conclude that $g$ acts WPD on $Z$.
\end{proof}

\subsection{The Gromov boundary of quotients} \label{section:boundary}

In this subsection, we discuss the boundary of the quotient graph we
obtain by coning off.  We first recall the work of Dowdall-Taylor.

Let $X$ be a Gromov hyperbolic graph and $Z$ a quotient graph obtained
by coning off a family $Y=\{Y_{i}\}_{i \in I}$ of quasi-convex subsets
of $X$.  Recall that the Gromov boundary of a Gromov hyperbolic space
$X$ is defined as the space of equivalence classes of quasi-geodesic
rays.  Given a quasi-geodesic ray $\gamma$ in $X$, let
$\gamma(\infty)\in\partial X$ denote the corresponding point in the
Gromov boundary $\partial X$.

Let $\mathrm{proj} \colon X \rightarrow Z$ denote the map given by the
inclusion $X\subset Z$.  To give a description of the Gromov boundary
of $Z$, Dowdall-Taylor defines a space $\partial_{Z}X$ by
$$\{\gamma(\infty)\in\partial X \mid\gamma\colon  \mathbb{R}_{+} \rightarrow X 
\text{ is a quasi-geodesic ray with } \mathrm{diam}_{Z} (\mathrm{proj}(\mathrm{Im}\gamma))=\infty\}.$$

Then the work of Dowdall-Taylor which we need here can be stated as
follows.

\begin{proposition}[see the proof of {\cite[Theorem
  4.8]{dowdall-taylor}}]\label{prop.DT}
Let $X$ be a Gromov hyperbolic space, and let $Y = \{ Y_i\}_{i \in I}$
be a collection of uniformly quasi-convex subsets of $X$.  Let $Z$ be
the space formed by coning off $X$ with respect to $Y$.  Then the
Gromov boundary of $Z$ is homeomorphic to $\partial_{Z}(X)$.
\end{proposition}

We shall define $\partial Y$ to be the union of the limit sets of each
$Y_i$, i.e. $\partial Y = \bigcup_{i \in I} \partial Y_i$.  This need
not be equal to the limit set of the union of the $Y_i$, which will
often be all of $\partial X$ if, for example, $G$ acts coarsely
transitively on $X$.

\begin{proposition}\label{prop.boundary}
Let $X$ be a Gromov hyperbolic space, and let $Y = \{ Y_i\}_{i \in I}$
be a collection of uniformly quasi-convex subsets of $X$.  Let $Z$ be
the space formed by coning off $X$ with respect to $Y$.  Let
$\partial Y := \bigcup_{i}\partial Y_{i}$.  Then the space
$\partial_{Z}(X)$ is contained in $\partial X\setminus\partial Y$.
\end{proposition}

\begin{proof}
Let $\gamma$ be a quasi-geodesic ray such that $\gamma(\infty)\in\partial Y_{i}$ for some $Y_{i}\in Y$.
Since $Y_{i}$ is quasi-convex,
we see that  $\gamma$ eventually fellow travels with $Y_{i}$, and thus the diameter of the image of $\gamma$ in $Z$ is finite.
\end{proof}

\subsection{Quotients and qi-embedded subgroups} \label{section:subgroups}

We consider qi(quasi isometrically)-embedded subgroups.
\begin{definition}
Let $X$ be a path-connected hyperbolic space on which $G$ acts isometrically. 
Let $H<G$ be a finitely generated subgroup. Then
\begin{itemize}
\item $H$  is said to {\em qi-embed} into $X$ if
$H\ni h\mapsto h(p)\in X$ defines a quasi-isometric embedding $H\rightarrow X$ for any $p\in X$.

\item {\em the limit set in $X$} of $H$ is the set of accumulation points in $\partial X$ of any orbit of $H$ in $X$.
\end{itemize}
\end{definition}

\begin{remark}
If $H$ qi-embeds into a geodesic Gromov hyperbolic space, then $H$ is also Gromov hyperbolic.
\end{remark}
We now characterize qi-embedded subgroups in the quotients.
\begin{theorem}\label{thm.qi-embed}
Let $X$, $Y$, $Z:=X_{Y}$, and $G$ as in Proposition \ref{prop.hyperbolic}.
Let $H$ be a finitely generated subgroup of $G$.
Then 
\begin{enumerate}
\item[(a)] $H$ qi-embeds in $X$ with limit set contained in $\partial X\setminus \partial Y$ in $X$
if and only if $H$ qi-embeds into $Z$
\item[(b)] If $H$ qi-embeds in $X$ with limit set contained in
$\partial X\setminus \partial Y$, then no element of infinite order in
$H$ preserves any $Y_{i}\in Y$ setwise.
\end{enumerate}
\end{theorem}

\begin{proof}[Proof of Theorem \ref{thm.asymmetric}]
Suppose $H$ qi-embeds into $Z$,
then $H$ must qi-embed into $X$.
Furthermore, if the limit set contains an element $\lambda$ in $\partial Y$, 
then there must be a sequence $(g_{i})_{i\geq 0}$ in $H$ such that 
$\{g_{i}x\}_{i}$ is quasi-isometrically embedded in $X$ and $g_{i} x\rightarrow \lambda\in\partial X$ for some $x\in X$.
As $\lambda\in\partial Y$, the orbit $g_{i} x$ must fellow travel with some $Y'\in Y$.
But this means that $g_{i} x$ is in a bounded region in $Z$ which contradicts the assumption that $H$ qi-embeds into $Z$.

Conversely, suppose $H$ qi-embeds in $X$ and the limit set is contained in $\partial X\setminus \partial Y$ in $X$.
Now we will prove that for every $D\geq 0$, there exists $N\geq 0$ such that
$d_{Z}(x, g x)\geq D$ whenever $|g|\geq N, g \in H$.
Suppose not, then there must be $D\geq 0$ and $(g_{i})_{i\geq 0}$ such that 
$d_{Z}(x, g_{i}x)<D$ and $|g_{i}|\rightarrow \infty$.
Then after taking subsequences, $g_{i}$ converges to a point in the boundary $\partial H$.
But then corresponding limit point in $\partial X$ must be in $\partial Y$, which contradicts our assumption.
Now (a) follows from \cite[Lemma 3.4]{dowdall-taylor}.

To prove (b), we note that if $H$ qi-embeds into $Z$, every element of infinite order in $H$ acts loxodromically on $Z$.
Then if $H$ contains an element of infinite order which preserves $Y'\in Y$, 
then the limit set in $X$ must contain a point in $\partial Y'$.
Thus (b) follows.
\end{proof}

\subsection{The compression body graph} \label{section:compression body}

Finally, we observe that these results apply in the case of $\C_D(S)$,
the electrification of the curve complex along the discs sets of
compression bodies.  These sets are unifomly quasiconvex
\cite{masur-minsky}, the electrification has infinite diameter, and
there are pseudo-Ansosov elements which act loxodromically on
$\C_D(S)$ \cite{maher-schleimer}.  It remains to verify condition (2)
from Theorem \ref{thm.WPD}, which is a corollary of the following
stability result.

\begin{theorem} \cite{maher-schleimer}*{Theorem
  4.2} \label{theorem:stability} %
Given a closed orientable surface $S$ and a constant $k$, there is a
constant $k' > k$ with the following property. Suppose that $\gamma$
is a geodesic ray in $\C(S)$, and $V_i$ is a sequence of compression
bodies such that, for all $i$, the segment $\gamma | [0, i]$ lies in a
$k$-neighborhood of $D(V_i)$. Then there is a constant $k'$ and a
non-trivial compression body $W$, contained in infinitely many of the
$V_i$, such that that $\gamma$ is contained in a $k'$-neighborhood of
$D(W)$.
\end{theorem}

We now verify condition (2) from Theorem \ref{thm.WPD} in this
setting.

\begin{corollary}
Let $S$ be a closed orientable surface, and let $g$ be a
pseudo-Ansosov map which acts loxodormically on $\C_D(S)$, and let
$\alpha$ be an axis for $g$ in $\C(S)$.  Then there is a constant $C$
such that for any disc set $D(V)$, the nearest point projection of
$D(V)$ ot $\alpha$ has diameter at most $C$.
\end{corollary}

\begin{proof}
Suppose there is a sequence of compression bodies $V_i$ such that
$\diam(\pi_\alpha(D(V_i)))$ tends to infinity.  As the $D(V_i)$ are
quasiconvex, and $g$ acts coarsely transitively on $\alpha$, there is
a constant $k$ such that we may translate each $D(V_i)$ by a power
$n_i$ of $g$ so that $\alpha | [0, t_i]$ is contained in a
$k$-neighbourhood of $g^{n_i} D(V_i)$, and $t_i$ tends to infinity as
$i$ tends to infinity.  Theorem \ref{theorem:stability} then implies
that there is a $k'$ such that $\alpha$ is contained in a
$k'$-neighbourhood of a single disc set $D(W)$.  This implies that the
image of $\alpha$ in $C_D(S)$ is finite, contradicting the fact that
$g$ acts loxodromically on $C_D(S)$.
\end{proof}

\section{The co-symmetric curve graph} \label{section:co-symmetric}

We introduce the co-symmetric curve graph.  Let $S = S_{g, n}$ be a
surface of finite type, where $g$ is the genus and $n$ is the number
of punctures of $S$.  If $3g - 3 + n > 0$, then the mapping class
group of $S$ contains pseudo-Anosov elements.  There is a finite list
of exceptional cases in which the mapping class group contains a
hyperelliptic involution which acts trivially on the curve graph and
the Teichm\"uller space of the surface, and for which every
pseudo-Anosov map is symmetric with respect to this involution.  We
shall call these special cases the \emph{exceptional} surfaces, and
they consist precisely of the surfaces $S_{0, 4}, S_{1, 1}, S_{1, 2}$
and $S_{2, 0}$, see for example \cite{farb-margalit}*{Section 12.1}.
We may exclude precisely these surfaces from the collection of
surfaces satisfying $3g - 3 + n > 0$, by replacing the condition
$3g - 3 + n > 0$ with the condition $2g + n > 4$.

By the word {\em curve} on $S$, we mean an isotopy class of essential
simple closed curve on $S$.  The curve graph, denoted
$\mathcal{C}(S)$, of $S$ is the graph each of whose vertices
corresponds to a curve, and two vertices are connected by an edge of
length $1$ if corresponding curves can be realized disjointly.  We
also need to consider orbifolds covered by $S$, so from now on any
finite cover may be an orbifold or branched cover of $S$. For such an
orbifold, we define the curve graph as the one for the surface we
obtain by puncturing all the orbifold points.  Given a finite covering
of surfaces, there is a map between the corresponding curve graphs.

\begin{definition}
Let $p\colon S\rightarrow S'$ be a finite covering.  We denote by
$\Pi\colon \mathcal{C}(S') \rightarrow \mathcal{C}(S)$ the associated
(one-to-finite) map on the curve graphs which is defined so that
$a\in\Pi(b)$ if and only if $p(a) = b$.  We call
$\Pi(\mathcal{C}(S'))$ the subspace of symmetric curves with respect
to $p\colon S\rightarrow S'$.
\end{definition}
By the work of Rafi-Schleimer, the subspaces of symmetric curves are quasi-isometrically embedded.
\begin{theorem}[\cite{RS}]\label{thm.RS}
The map $\Pi\colon \mathcal{C}(S')\rightarrow\mathcal{C}(S)$ is a quasi-isometric embedding with quasi-isometry constants depending
only on $S$ and the degree of the finite covering $p\colon S\rightarrow S'$.
\end{theorem}

\begin{remark}\label{rmk.qi-->qc}
In a Gromov hyperbolic space, every quasi-isometrically embedded subspace is quasi-convex and the quasi-convexity constant depends only on the quasi-isometry constants and the hyperbolicity constant.
\end{remark}
The main object we consider in this section is the following.
\begin{definition}
The {\em co-symmetric} curve graph $\CS(S)$ is the graph obtained by coning off the family 
$\Pi:=\{\Pi(\mathcal{C}(S'))\}$ of  subspaces of symmetric curves, where we consider all finite (possibly orbifold) coverings of type $p\colon S\rightarrow S'$.
\end{definition}

\begin{proposition}
The co-symmetric curve graph $\CS(S)$ is Gromov hyperbolic and
$\mathrm{MCG}(S)$ acts isometrically on $\CS(S)$.
\end{proposition}

\begin{proof}
Given a finite covering $p\colon S\rightarrow S'$, the map $p\circ g^{-1}$ is also a finite covering for any $g\in\mathrm{MCG}(S)$.
Therefore the family $\Pi$ is setwise $\mathrm{MCG}(S)$-invariant.
Furthermore, as there is a lower bound of Euler characteristics of $2$-dimensional orbifolds, the degree of the coverings from $S$ is bounded.
Hence by Theorem \ref{thm.RS} and Remark \ref{rmk.qi-->qc}, the quasi-convexity constants of the elements in the family $\Pi$ is bounded.
Hence, we get the conclusion by Proposition \ref{prop.hyperbolic}.
\end{proof}

Note that we have not yet eliminated the possibility that $\CS(S)$ has
bounded diameter.  To discuss further properties of $\CS(S)$, the
following terminology is useful.

\begin{definition}
Let $A$ be a mapping class, a foliation, or a lamination.  Then $A$ is
said to be {\em symmetric} if it is a lift with respect to some finite
covering $p \colon S \to S'$.  If $A$ is not symmetric, it is said to
be {\em asymmetric}.
\end{definition}

In subsection \ref{section:action}, we discuss the $\mathrm{MCG}(S)$-action on $\CS(S)$.
By using Theorem \ref{thm.WPD}, we will prove the following theorem which is a detailed version of Theorem \ref{thm.main} for $\CS(S)$.
\begin{theorem}\label{thm.cosymmetric}
Let $S$ be an orientable surface of genus $g$ with $n$ punctures.
Suppose $2g + n > 4$.
For the action of $\mathrm{MCG}(S)$ on $\CS(S)$, we have
\begin{enumerate}
\item[(a)]  
$g\in\mathrm{MCG}(S)$ is loxodromic if and only if $g$ is asymmetric and pseudo-Anosov, 
\item[(b)] the action of $\mathrm{MCG}(S)$ on $\CS(S)$ is WPD.
\end{enumerate}
In particular, the action of $\mathrm{MCG}(S)$ on $\CS(S)$ is non-elementary.
\end{theorem}

We now give a characterization of $\partial\CS(S)$.
Recall that by the work of Klarreich \cite{klarreich},
the boundary of $\mathcal{C}(S)$ is identified with the space of ending laminations, denoted $\mathcal{EL}(S)$.
We call an ending lamination {\em symmetric} if 
it is a lift of an ending lamination with respect to  a finite covering $p\colon S\rightarrow S'$.
An ending lamination is called {\em asymmetric} if it is not symmetric.
Let $\mathcal{AEL}(S)$ denote the subspace of $\mathcal{EL}(S)$ 
consisting of all asymmetric ending laminations.
Then by Proposition \ref{prop.DT} and Proposition \ref{prop.boundary} we have the following.
\begin{theorem}
The Gromov boundary $\partial\CS(S)$ is contained in $\mathcal{AEL}(S)$.
\end{theorem}

In \cite{farb-mosher}, Farb-Mosher defines the notion of convex cocompact subgroups of $\mathrm{MCG}(S)$ 
as an analogue of convex cocompact Kleinian groups.
Instead of giving the definition, we regard the following theorem due to Kent-Leininger as a definition of convex cocompactness.
\begin{theorem}[{\cite[Section 7]{kent-leininger}}]
A finitely generated subgroup $H<\mathrm{MCG}(S)$ is convex cocompact if and only if $H$ qi-embeds into $\mathcal{C}(S)$.
\end{theorem}

We call a subgroup $H<\mathrm{MCG}(S)$ {\em purely asymmetric} if every element in $H$ of infinite order is asymmetric.
This is equivalent to saying every element in $H$ of infinite order does not preserve any symmetric subspaces.
Then by Theorem \ref{thm.qi-embed},
we have:
\begin{theorem}\label{thm.asymmetric}
Let $S$ be as in Theorem \ref{thm.cosymmetric}.
Let $H$ be a finitely generated subgroup of $\mathrm{MCG}(S)$.
Then 
\begin{enumerate}
\item[(a)] $H$ is convex cocompact with every element in the limit set asymmetric in $\mathcal{C}(S)$
if and only if $H$ qi-embeds into $\CS(S)$.
\item[(b)] If $H$ satisfies the equivalent condition in (a),  then $H$ is purely asymmetric.
\end{enumerate}
\end{theorem}

For any group with non-elementary action by isometries on a Gromov
hyperbolic space $X$, Taylor-Tiozzo \cite{TT} proved that random
subgroups given by independent random walks are qi-embedded free
groups.  Putting Theorem \ref{thm.cosymmetric}, Theorem
\ref{thm.asymmetric} and \cite{TT} together we get the following.
\begin{corollary}
There are infinitely many purely asymmetric subgroups in $\mathrm{MCG}(S)$.
\end{corollary}

\subsection{The $\MCG(S)$ action on $\CS(S)$ is WPD} \label{section:action}

In this subsection, we prove Theorem \ref{thm.cosymmetric}.
First, we prepare the following.
\begin{lemma}\label{lem.wpd-symmetric}
Let $g\in\mathrm{MCG}(S)$ be an asymmetric pseudo-Anosov element with quasi-geodesic axis $\alpha$.
We denote by $\pi_{\alpha}$ the nearest projection map to $\alpha$ on $\mathcal{C}(S)$.
Then there exists $C>0$ such that for any finite covering $p\colon S\rightarrow S'$ and the corresponding symmetric subspace $\Pi(\mathcal{C}(S'))$,
we have $\mathrm{diam}_{\mathcal{C}}(\pi_{\alpha}(\Pi(\mathcal{C}(S'))))<C$, where $\diam_{\mathcal{C}}$ is the diameter in $\mathcal{C}(S)$.
\end{lemma}

A key ingredient of a proof of Lemma \ref{lem.wpd-symmetric} is the
uniform finiteness of the number of parallel symmetric subspaces,
which is proved in \cite{Mas}.

\begin{lemma}[{\cite[Lemma 4.5.]{Mas}}]\label{lem.finiteparallel}
Let $p\colon S\rightarrow S'$ be a finite covering and denote the
corresponding symmetric subspace by $\Pi(\mathcal{C}(S'))$.  Then for
any $D_0>0$, there exist $D_1, D_2>0$ which depend only on $S$ and
$D_0$ such that for any $a,b\in \mathcal{C}(S)$ with
$d_\mathcal{C}(a,b)>D_1$, the number of distinct elements in
$$\{g\Pi(\mathcal{C}(S'))\mid
d_\mathcal{C}(a,g\Pi(\mathcal{C}(S')))<D_0 \text{ and }
d_\mathcal{C}(b,g\Pi(\mathcal{C}(S')))<D_0\}$$
is bounded from above by $D_{2}$.  Here we count the number of
distinct images i.e. if
$g_{1}\Pi(\mathcal{C}(S')) = g_{2}\Pi(\mathcal{C}(S'))$ as subsets, we
just count one time.
\end{lemma}

We need the following corollary to Lemma \ref{lem.finiteparallel} to
prove Lemma \ref{lem.wpd-symmetric}.  This corollary may be of
independent interest.

\begin{corollary}\label{cor.asym}
Let $g\in\mathrm{MCG}(S)$ be pseudo-Anosov.
Then $g$ is symmetric if and only if one of its stable or unstable foliation is symmetric.
\end{corollary}
\begin{proof}
Let $\lambda_{+}(g)$ and $\lambda_{-}(g)$ denote the stable foliation and the unstable foliation of $g$ respectively.
If $g$ is symmetric, then both $\lambda_{+}(g)$ and $\lambda_{-}(g)$ must be symmetric.
To prove the converse, we first remark that in \cite[Section 3]{Mas-q},
it is proved that if both $\lambda_{+}(g)$, $\lambda_{-}(g)$ are symmetric with respect to the same covering, then $g$ must be symmetric.
Hence it suffices to prove that if $\lambda_{-}(g)$ is symmetric with respect to a finite covering $p\colon S\rightarrow S'$, 
then $\lambda_{+}(g)$ must be symmetric with respect to $p$.
Suppose $\lambda_{-}(g)\in \partial\Pi(\mathcal{C}(S'))$.
Since $\lambda_{-}(g)$ is fixed by $g$ and $g^{n}(a)\rightarrow \lambda_{+}(g)$ for any $a\in\mathcal{C}(S)$, 
translates $g^{n}(\Pi(\mathcal{C}(S')))$ for large enough $n>>1$ must fellow travel each other.
If $\{g^{n}(\Pi(\mathcal{C}(S')))\}_{n\in \mathbb{N}}$ are all distinct, then this contradicts Lemma \ref{lem.finiteparallel}.
Hence for some $i\not=j$, we have $g^{i}(\Pi(\mathcal{C}(S'))) = g^{j}(\Pi(\mathcal{C}(S')))$.
Then we see that $g^{i-j}(\Pi(\mathcal{C}(S'))) = \Pi(\mathcal{C}(S'))$ and hence $\lambda_{+}(g)\in\partial\Pi(\mathcal{C}(S'))$.
This completes the proof.
\end{proof}

Now we prove Lemma \ref{lem.wpd-symmetric}.
\begin{proof}[Proof of Lemma \ref{lem.wpd-symmetric}]
Suppose the contrary that there is a sequence $\Pi_{n}$ of symmetric subspaces such that $\diam_\mathcal{C}(\pi_{\alpha}(\Pi_{n}))\rightarrow \infty$.
Since in any finitely generated group, the number of subgroups of a given index is finite,
the number of coverings from $S$ up to conjugacy is finite.
Hence after taking a subsequence, we may suppose that $\Pi_{n}=g_{n}\Pi(\mathcal{C}(S'))$ for some $g_{n}\in\mathrm{MCG}(S)$ and some fixed symmetric subspace $\Pi(\mathcal{C}(S'))$.
We fix a point $p\in\alpha$.
By applying suitable power of $g$ to $\Pi_{n}$ we may suppose that all $\pi_{\alpha}(\Pi_{n})$ are coarsely centered at $p$.
By Gromov hyperbolicity of $\mathcal{C}(S)$, if $\diam_\mathcal{C}(\pi_{\alpha}(\Pi_{n}))$ is greater than a constant which depends only on the hyperbolicity constant,
then $\Pi_{n}$ must fellow travel with $\alpha$.
By Corollary \ref{cor.asym}, the stable and unstable foliations of $g$ are asymmetric, which implies that $\{\Pi_{n}\}$ must contain infinitely many distinct elements.
Thus we get a contradiction to Lemma \ref{lem.finiteparallel}.
\end{proof}

Finally, we verify the existence of asymmetric pseudo-Anosov maps.  We
remark that this result does not hold for the exceptional surfaces, as
every map is a lift with respect to the covering given by a
hyperelliptic involution, and we now enumerate them for the
convenience of the reader.  Given an orbifold cover
$p \colon S \to S'$, we shall write $S'$ as
$S'_{g'}(d_1, \ldots d_{n'})$, where $g'$ is the genus of $S'$, $n'$
is the number of orbifold points, and the $d_i \ge 2$ are the orders
of the orbifold points, so each orbifold points has metric angle
$2\pi/d_i$.  We allow $d_i = \infty$, corresponding to a puncture.
With this notation the exceptional cases are:

\begin{itemize}

\item $p \colon S_{0, 4} \to S'_0(2, 2, \infty, \infty)$

\item $p \colon S_{1, 1} \to S'_0(2, 2, 2, \infty)$

\item $p \colon S_{1, 2} \to S'_0(2, 2, 2, 2, \infty)$

\item $p \colon S_{2, 0} \to S'_0(2, 2, 2, 2, 2, 2)$

\end{itemize}

Recall that the dimension of the Teichm\"uller space of an orbifold
$S'_{g'}(d_1, \ldots d_{n'})$ is $6g' - 6 + n'$, and so the dimension
only depends on the number of orbifold points, not on their orders.
This is because any two orbifold points have small neighbourhoods
which are conformally equivalent.  Therefore the set of conformal
structures on $S'_{g'}(d_1, \ldots , d_{n'})$ is equal to the set of
conformal structures on $S_{g', n'}$, and due to work of Picard
\cite{picard} in the closed case, and Heins \cite{heins} for the
punctured case, given a choice of order $d_i$ at each orbifold point,
there is a unique hyperbolic metric with cone points of that order in
a given conformal class.

The following result is well-known, but we provide the details for
convenience.

\begin{proposition} \label{prop:exceptional} %
Let $S$ be a surface of genus $g$ with $n$ punctures, with
$2g + n > 4$, and let $p \colon S \to S'$ be an orbifold cover.  Then
$\text{dim}( \T(S') ) < \text{dim}( \T(S) )$.
\end{proposition}

\begin{proof}
The hyperbolic metric on $S'$ lifts to a hyperbolic metric on $S$, and
this gives an isometric embedding $\T(S') \to \T(S)$, see for example
\cite[Section 7]{RS}.  In particular, if the dimension of $\T(S')$ is
equal to the dimension of $\T(S)$, then the image of $\T(S')$ is equal
to $\T(S)$.

First consider the case in which the cover is regular, with finite
deck transformation group $F < \MCG(S)$.  The image of $\T(S')$ in
$\T(S)$ is the set of points fixed by all elements of $F$, and so the
fixed point set of $F$ is all of $\T(S)$.  However, the mapping class
group $\MCG(S)$ acts faithfully on $\T(S)$, except for the finite list
of exceptional cases $S_{0, 4}, S_{1, 1}, S_{1, 2}$ and $S_{2, 0}$ in
which there are hyperelliptic involutions which act trivially on
$\T(S)$, see for example \cite{farb-margalit}*{Section 12.1}.
However, these are precisely the surfaces we have excluded with the
condition $2g + n > 4$.

Now suppose that the cover $p \colon S \to S'$ is not regular.  Let
$S'$ have genus $g'$, and $n'$ orbifold points.  The generalization of
Royden's Theorem \cite{royden} due to Earle and Kra \cite{earle-kra}
states that if $\T(S)$ is isometric to $\T(S')$ then $S$ is
homeomorphic to $S'$, unless $2g + n \le 4$, but this is precisely the
exceptional cases we have excluded.  This implies that $g = g'$ and
$n = n'$.  We now show that in fact there can be no such orbifold
cover with these properties.

Recall that the Euler characteristics of the two surfaces are given by
\[ \chi(S) = 2 - 2g - n \quad \text{ and } \quad \chi(S') = 2 - 2g' -
\sum_{i=1}^{n'}(1 - \frac{1}{d_i}).  \]
The order of an orbifold point satisfies $d_i \ge 2$, and so as
$g = g'$ and $n = n'$:
\[ \chi(S) - \chi(S') \ge  - \frac{1}{2} n.   \]
Euler characteristic is multiplicative under covers, so
$\chi(S') = \frac{1}{d} \chi(S')$, where $d$ is the degree of the
covering map:
\[ \chi(S) - \frac{1}{d}\chi(S) \ge  - \frac{1}{2} n.   \]
This gives:
\[ (1 - \frac{1}{d})(2 - 2g - n) \ge  - \frac{1}{2} n.   \]
Which we may rewrite as:
\[ \left( 1 - \frac{1}{d} \right) (2 - 2g) \ge \left( \frac{1}{2} -
\frac{1}{d} \right) n .  \]
As we are assuming that the cover is irregular, $d \ge 3$.  The only
possible non-negative integer solutions have $g = 0$ or $g = 1$.  If
$g = 1$, then $n = 0$, but the torus does not have a hyperbolic
metric.  If $g = 0$, then
\[  n \le \frac{4d - 4}{d - 2}. \]
For $d \ge 7$ this implies that $n < 5$, so the only possibility is
the exceptional surface $S_{0, 4}$.  The remaining finite list of
possibilities consists of:
\[ S_{0, 8}, d = 3, \quad S_{0, 7}, d=3, \quad S_{0, 6}, d=3, 4, \quad
S_{0, 5}, d = 3, 4, 5, 6 \]
We now consider these in turn.

\begin{itemize}

\item[$d=3$] First consider the case in which the cover is degree
$d = 3$, and the surface $S$ is $S_{0, n}$, with $5 \le n \le 8$.  The
Euler characteristic of $S$ is $\chi(S) = 2 - n$, so
$\chi(S') = \frac{2 - n}{3}$.  Let $S'$ have $a$ orbifold points of
order $3$ and $n - a$ orbifold points of order $\infty$.  Then
\[ \chi(S') = 2 - \frac{2}{3}a - (n-a) = \frac{2-n}{3}, \]
which implies that $a = 2n - 4$.  However, $a \le n$, which implies
that $n \le 4$, a contradiction.

\item[$d=4$] Now consider the case in which the cover is degree
$d = 4$, and the surface $S$ is $S_{0, n}$, with $5 \le n \le 6$.  The
Euler characteristic of $S$ is $\chi(S) = 2 - n$, so
$\chi(S') = \frac{2 - n}{4}$.  Let $S'$ have $a$ orbifold points of
order $2$, $b$ orbifold points of order $4$ and $n - a - b$ orbifold
points of order $\infty$.  Then
\[ \chi(S') = 2 - \frac{1}{2}a - \frac{3}{4}b - (n- a -b) =
\frac{2-n}{4}, \]
which we may rewrite as
\[  2a + b = 3n - 6.  \]

When $n = 5$ we obtain $2a + b = 9$, but $a+b \le 5$, so the only
possible solution with non-negative integers is $a = 4$ and $b=1$, but
this means that $S'$ has no punctures, a contradiction.

When $n = 6$ we obtain $2a + b = 12$, but $a+b \le 6$, so the only
possible solution with non-negative integers is $a = 6$ and $b=0$, but
again, this means that $S'$ has no punctures, a contradiction.

\item[$S = S_{0, 5}$] Finally, consider possible quotients of the
surface $S = S_{0, 5}$, with degree $d = 5$ or $d = 6$.

When $d = 5$, as $\chi(S) = -3$, $\chi(S') = -3/5$.  Let $S'$ have $a$
orbifold points of order $5$ and $5 - a$ orbifold points of order
$\infty$.  Then
\[ \chi(S') = 2 - \frac{4}{5}a - (5-a) = - \frac{3}{5}, \]
which implies that $a = 12$, a contradiction.

When $d = 6$, as $\chi(S) = -3$, $\chi(S') = -1/2$.  Let $S'$ have $a$
orbifold points of order $2$, $b$ orbifold points of order $3$, $c$
orbifold points of order $6$ and $5 - a - b - c$ orbifold points of
order $\infty$.  Then
\[ \chi(S') = 2 - \frac{1}{2}a - \frac{2}{3}b - \frac{5}{6}c -
(5-a-b-c) = - \frac{1}{2}, \]
which we may rewrite as:
\[ 3a + 2b + c = 15.  \]
However, each of $a, b$ and $c$ is at most $5$, so the only possible
solution with non-negative integers is $a = 5$, $b = 0$ and $c = 0$.
This implies that the number of punctures for $S'$ is also zero, a
contradiction.
\end{itemize}
This completes the proof of Proposition \ref{prop:exceptional}. 
\end{proof}

\begin{proposition}\label{prop.asymmetric pA}
Let $S$ be a surface of genus $g$ and $n$.  Suppose $2g + n > 4$.
Then there exists an asymmetric pseudo-Anosov element in
$\mathrm{MCG}(S)$.
\end{proposition}
\begin{proof}
An orbifold cover $p \colon S \to S'$ gives an isometric embedding of
$\T(S')$ in $\T(S)$, and as we are not in one of the exceptional
cases, by Proposition \ref{prop:exceptional}, the dimension of
$\T(S')$ is strictly less than the dimension of $\T(S)$.

There is a finite index subgroup $H < \MCG(S')$ consisting of mapping
classes which lift to maps of $S$.  By abuse of notation, we shall
also refer to the corresponding subgroup in $\MCG(S)$ as $H$, and this
subgroup preserves the image of $\T(S')$ in $\T(S)$.

Let $B(x, r)$ be a ball in $\T(S)$, then there is an $\e > 0$ such
that $B(x, r)$ lies in the $\epsilon$-thick part of $\T(S)$.  Pick a
basepoint on $\T(S')$.  The subgroup $H$ acts coarsely transitively on
the $\e$-thick part of $\T(S')$, and so if some translate $g\T(S')$,
for $g$ in $\MCG(S)$, intersects $B(x, r)$, then there is a translate of
the basepoint a bounded distance from $x$.  As $\MCG(S)$ acts properly on
$\T(S)$ there are only finitely many distinct translates of $\T(S')$
intersecting $B(x, r)$.

As there are only finitely many quotients $p \colon S \to S'$ up to
covering space isomorphism, there are only finitely many isometrically
embedded symmetric subsets $\T(S')$ intersecting $B(x, r)$.  As we are
in the non-exceptional case, they all have strictly smaller dimension
than $\T(S)$.  In particular, there is a ball
$B(x', r') \subseteq B(x, r)$ in $\T(S)$ disjoint from all of the
symmetric subsets (i.e. the union of the $gT(S')$ over all $g$ and all
covers).

Periodic geodesics are dense in $\T(S)$, see for example
\cite{eskin-mirzakhani}*{Section 6}.  Therefore, there is a
pseudo-Anosov mapping class with geodesic axis intersecting
$B(x', r')$, and this pseudo-Anosov mapping class is asymmetric, as
required.
\end{proof}

\begin{proof}[Proof of Theorem \ref{thm.cosymmetric}]
By Lemma \ref{lem.wpd-symmetric}, we see that the assumption (2) of Theorem \ref{thm.WPD} is satisfied for
$X = \mathcal{C}(S)$, $Y=\{\Pi(\mathcal{C}(S'))\}$ and an asymmetric pseudo-Anosov $g$.
Then $Z=\CS(S)$.
As every pseudo-Anosov element is WPD on $\mathcal{C}(S)$ \cite[Proposition 11]{bestvina-fujiwara}, 
every asymmetric pseudo-Anosov is also WPD on $\CS(S)$, especially it is loxodromic.
On the other hand, note that if an element in $\mathrm{MCG}(S)$ acts loxodromically on $\CS(S)$,
then it must act loxodromically on $\mathcal{C}(S)$.
Since every symmetric pseudo-Anosov element must fix corresponding symmetric subspaces, it can not be loxodromic.
Thus (a) follows.

We now prove (b).  If $2g + n > 4$, then $\mathrm{MCG}(S)$ contains a
free group of rank $2$, and in particular it is not virtually cyclic.
Then by Proposition \ref{prop.asymmetric pA}, at least one loxodromic
element exists.  We have already seen that every loxodromic element
acts WPD, hence the action of $\mathrm{MCG}(S)$ on $\CS(S)$ is strongly WPD.
That the action is non-elementary is the consequence of Proposition
\ref{prop.non-elementary}.
\end{proof}

\section{Non-maximal train tracks} \label{section:non-maximal}

We say a pair of curves $a$ and $b$ on a surface is \emph{filling} if
when put in minimal position, there is no essential closed curve
disjoint from $a \cup b$.  In particular, this implies that every
complementary region is a disc containing at most one puncture.  We
say a pair of curves is \emph{maximally filling} if when put in
minimal position every complementary region is a square, hexagon or a
bigon containing a single puncture.

We say a pair of curves $a$ and $b$ is \emph{$K$-maximally filling}, if
for any pair of curves $a'$ and $b'$ with $d_{\C(S)}(a, a') \le K$ and
$d_{\C(S)}(b, b') \le K$, then $a'$ and $b'$ are maximally filling.

In this section we will consider surfaces $S_{g, n}$ which support a
pair of curves which are maximally filling, and which also support a
pair of curves which fill the surface, but are \emph{not} maximally
filling.  This consists of all surfaces which support a pseudo-Anosov
map, with the exceptions of $S_{0, 4}$ and $S_{1, 1}$.  The
four-punctured sphere $S_{0, 4}$ has the property that all filling
curves are maximally filling.  The punctured torus $S_{1, 1}$ has the
property that there are no maximally filling pairs of curves, as any
pair of curves in minimal position may be realised as a pair of
geodesics in a flat metric on the torus with a marked point, and the
marked point is then contained in a square, not a bigon.  We shall
therefore consider surfaces $S_{g, n}$ with $3g + n > 4$, as this
condition includes all surfaces which support a pseudo-Anosov map,
excluding $S_{0, 4}$ and $S_{1, 1}$.

\begin{remark} \label{remark:stratum} %
In the case that the surface $S$ is either $S_{0, 4}$ or $S_{1, 1}$,
the Teichm\"uller space $\T(S)$ has complex dimension $1$, so there is
a single stratum of quadratic differentials, and all the pseudo-Anosov
maps have invariant foliations with the same collection of
singularities.  In the case of $S_{0, 4}$, the invariant foliation is
the quotient of an invariant foliation for an Anosov map on a torus by
the hyperelliptic involution, with the order two cone points replaced
by punctures.  This foliation then has exactly four $1$-prong
singularities, and so Theorem \ref{theorem:gadre-maher} holds in this
case.  However, in the case of $S_{1, 1}$, Theorem
\ref{theorem:gadre-maher} does not hold, as the invariant foliations
are preserved by the hyperelliptic involution, which fixes the
puncture, and so the foliations cannot have $1$-prong singularities at
the puncture.
\end{remark}

\begin{lemma} \label{lemma:maximal} Let $S = S_{g, n}$ be a surface of
finite type of genus $g$ with $n$ punctures, with $3g + n > 4$.  Let
$g$ be a pseudo-Anosov element on $S$ with maximal invariant
laminations and let $\alpha$ be an axis for $g$.  Then for any $K$
there is an $R$ such that any two curves on $\alpha$ distance at least
$R$ apart are $K$-maximally filling.
\end{lemma}

We shall prove this result in this section.  We start by reviewing
some useful properties of train tracks.

\subsection{Train tracks}

We briefly recall some of the results we use about train tracks on
surfaces. For more details see for example Penner and Harer
\cite{penner-harer}.

Recall that a train track on a surface $S$ is a smoothly embedded
graph, such that the edges at each vertex are all mutually tangent,
and there is at least one edge in each of the two possible directed
tangent directions. Furthermore, there are no complementary regions
which are nullgons, monogons, bigons or annuli.  The vertices are
commonly referred to as \emph{switches} and the edges as
\emph{branches}.  We will always assume that all switches have valence
at least three.  A trivalent switch is illustrated below in Figure
\ref{fig:switch}.

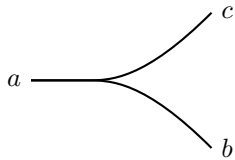
\begin{figure}[H]
\begin{center}
\begin{tikzpicture}[scale=0.3]

\draw [thick] (20,13) node [left] {$a$} -- 
              (22.5, 13) .. controls (23.5,13) and (25,13) .. 
              (28,10) node [right] {$b$};
\draw [thick] (20,13) -- 
              (22.5, 13) .. controls (23.5,13) and (25,13) .. 
              (28,16) node [right] {$c$};

\end{tikzpicture}
\end{center}
\caption{A trivalent switch for a train track.} \label{fig:switch}
\end{figure}

An assignment of non-negative numbers to the branches of $\tau$, known
as \emph{weights}, satisfies the \emph{switch equality} if the sum of
weights in each of the two possible directed tangent directions is
equal: that is $a = b + c$ in Figure \ref{fig:switch} above.

An assignment of non-negative numbers to the branches of $\tau$, known
as \emph{weights}, satisfies the \emph{switch equality} if the sum of
weights in each of the two possible directed tangent directions is
equal: that is $a = b + c$ in Figure \ref{fig:switch} above.
A train track with integral weights defines a simple
closed multicurve on the surface, and we say that this curve is
\emph{carried} by $\tau$.  We shall write $\C(\tau)$ for the subset of
the curve complex consisting of the simple closed curves carried by
$\tau$.

More generally, for weights with non-integral values, a weighted train
track defines a measured foliation on the surface.  We say that the
corresponding foliation is \emph{carried} by the train track.  We
shall write $\mathcal{MF}(S)$ for the space of measured laminations on
the surface $S$, and $P(\tau)$ for the set of measured foliations
carried by $\tau$.  The set $P(\tau)$ is projectively invariant, and
we shall write $\overline{P(\tau)}$ for its projectivization in
$\PMF(S)$, the space of projective measured foliations on the surface.
The set $\overline{P(\tau)}$ is a polytope in $\PMF(S)$.  Let
$V(\tau)$ be the set of vertices of $\overline{ P(\tau) }$. Every
$v \in V(\tau)$ gives a \emph{vertex cycle}: a simple closed curve,
carried by $\tau$, that puts weight at most two on each branch of
$\tau$.  By abuse of notation, we shall also write $V(\tau)$ for the
corresponding set of simple closed curves in $\C(S)$.

We say a train track $\tau$ is \emph{maximal} if every complementary
region is a triangle or a monogon containing a single puncture.  Every
surface which contains a maximally filling pair of curves also contains
a maximal train track, as given a choice of orientations on the
curves, one may smooth the intersections compatibly to produce a
maximal train track.

\subsection{Non-classical interval exchange transformations}

Train tracks are closely related to non-classical interval exchanges,
which we now describe, see for example \cite{boissy-lanneau}.

\begin{definition} A \emph{non-classical interval exchange} consists
of the following:

\begin{enumerate}

\item Given positive numbers $\delta$ and $\e$, let $B_0$ be a
Euclidean rectangle $[0, \delta] \cross [0, \e]$ which we shall call
the \emph{base rectangle}.  We will refer to the sides of length $\e$
as the \emph{vertical} sides and the sides of length $\delta$ as the
\emph{horizontal} sides.  We shall call $\{ 0 \} \cross [0, \e]$ the
\emph{initial vertical side} and $\{ \delta \} \cross [0, e]$ the
\emph{terminal vertical side}.  We shall call the horizontal side
$[0, \delta] \cross \{ \e \}$ the \emph{upper side} and label it
$I_+$, and the other side $[0, \delta] \cross \{ 0 \}$ the \emph{lower
  side} and label it $I_-$.

\item Let $B_1, \ldots B_n$ be a finite collection of metric Euclidean
rectangles, which we shall call \emph{bands}.  Each band $B_i$ has one
pair of opposite sides called the horizontal sides, and one pair
called the vertical sides.  The length of the horizontal sides of the
band $B_i$ is called the \emph{width} $w_i$ of the band $B_i$, and the
total width of the bands $\sum w_i $ is equal to $L$.

\item For each horizontal side of a band there is a Euclidean isometry
from the horizontal side to the disjoint union $I_+ \sqcup I_-$, with
the following properties:

\begin{enumerate}

\item The images of the interiors of any two horizontal sides are
disjoint.

\item The quotient space obtained from gluing a band to the base
rectangle along its two horizontal sides is orientable,
i.e. homeomorphic to an annulus and not a M\"obius band.

\end{enumerate}

If a band has one horizontal side mapped to $I_+$, and the other
mapped to $I_-$, then we call it an \emph{orientation preserving
  band}.  Otherwise we call it an \emph{orientation reversing band}.

\end{enumerate}

\end{definition}

Every generalized interval exchange transformation gives rise to a
measured train track with a single vertex, by collapsing the base
rectangle to a vertex, and each band to a edge of weight $w_i$.

In the description above we attached the bands to the horizontal sides
of the rectangle, but we could instead have attached them all to the
vertical sides.

\subsection{Quadratic differentials and transversals}

A pseudo-Anosov map $g$ preserves a unique geodesic axis $\gamma$ in
Teichm\"uller space.  Choose a unit speed parameterization $\gamma_t$
such that $g(\gamma_0) = \gamma_t$ for some positive $t$.  Let $q$ be
the quadratic differential at $\gamma_0$ determined by the geodesic
$\gamma$.  The quadratic differential $q$ determines a flat structure
on $S$ which we shall denote $S_q$.  Let $F_+$ be the vertical
measured foliation determined by $S_q$ , whose projectivization
$\overline{F_+}$ is the stable invariant projective measured foliation
for $g$.  Similarly, let $F_-$ be the horizontal measured foliation
for $S_q$, whose projectivization $\overline{F_-}$ is the unstable
invariant foliation for $g$.

A \emph{transversal} $t$ for a foliation $F$ in $S$ is an embedded arc
in $S$ which is disjoint from the singularities of $F$, and is
transverse to the leaves of $F$.  We say a transversal for $F_+$ is
\emph{horizontal} if it is a horizontal geodesic in the flat metric
$S_q$ corresponding to $q$.  Similarly we say a transversal for $F_-$
is \emph{vertical} if it is a vertical geodesic.

The transversal $t$ determines a non-classical interval exchange.
This may be thought of as arising by cutting the surface along the
singular flow lines, and then completing each maximal open interval of
non-singular flow lines by adding a pair of non-singular edges at each
end.  We now give a detailed description of this construction for
horizontal transversals for the vertical foliation $F_+$.

Let $t$ be a horizontal transversal for the vertical foliation $F_+$.
As $t$ is a horizontal geodesic disjoint from the singular set of $F$,
there is an $\e > 0$ such that the Euclidean rectangle
$t \cross [-\e, \e]$ is also disjoint from the singular set.  We shall
choose this to be the base rectangle $B_0$ in the non-classical
interval exchange map.  Recall that one horizontal side of the base
rectangle is called $I_+$ and the other horizontal side is called
$I_-$.

A \emph{flow line} $\ell$ is the closure of a connected component of a
leaf of the foliation $F$ in $S \setminus B_0$.  A flow line is
\emph{non-singular} if it is a properly embedded arc with distinct
endpoint in the horizontal boundary of $B_0$.  A flow line is
\emph{singular} if it contains a singular point of the foliation $F$.
A \emph{tripod} is the topological space homeomorphic to the connected
and simply connected graph consisting of three edges meeting at a
common vertex.  The foliation $F_+$ contains finitely many singular
points, all of which are trivalent, so there are finitely many
singular flow lines, which are tripods, properly embedded with
distinct endpoints in the horizontal boundary of $B_0$.  All other
flow lines are non-singular, and this determines a first return map on
$I_+ \sqcup I_-$, defined on all but finitely many points, which is an
isometry on connected intervals for which it is defined.  Furthermore,
as $S$ is orientable, any subsurface consisting of $B_0$ and any
collection of non-singular flow lines is orientable, so if we choose a
band for each maximal subinterval of $I_+ \sqcup I_-$ on which the
first return map is defined, this is a non-classical interval exchange
transformation.

The same construction applies to a vertical transversal for the
horizontal foliation $F_-$, but with the vertical and horizontal
directions swapped.

\subsection{Rauzy induction}

We say a transversal $t = [t_0, t_1]$ for a foliation $F$ is
\emph{admissible} if it determines a non-classical interval exchange
such that the terminal point $t_1$ is the endpoint of a singular flow
line.  In this case, let $t_1' \in [t_0, t_1)$ be the closest endpoint
of a singular flow line to $t_1$.  The transversal $[t_0, t'_1]$ is a
subset of $[t_0, t_1]$ and determines a new non-classical interval
exchange, again with the property that the terminal endpoint meets a
singular flow line.  An example of this is illustrated in Figure
\ref{fig:rauzy}.

\begin{figure}[H]
\begin{center}
\begin{tikzpicture}[scale=0.8]

\tikzstyle{point}=[circle, draw, fill=red, inner sep=1pt]

\draw [rounded corners] (12, 1) -- (12, 3) -- (9, 3) -- (9, 1);

\draw [rounded corners] (11, 1) -- (11, 2) -- (10, 2) -- (10, 1);

\draw [rounded corners] (12, 0) -- (12, -1) -- (7, -1) -- (7, 0);

\draw [rounded corners] (11.5, 0) -- (11.5, -0.5) -- (7.5, -0.5) --
(7.5, 0);

\draw [thick, red] (6, 0) node [black, below right] {$I_-$} -- (12, 0)
-- (12, 1) -- (6, 1) node [black, above right] {$I_+$} -- cycle;

\draw [thick, red, dashed] (11.5, 0) -- (11.5, 1);

\draw (6, 0) node [red, point, label=above right:${t_0}$] {};

\draw (11.5, 0) node [red, point, label=above left:${t'_1}$] {};

\draw (12, 0) node [red, point, label=above right:${t_1}$] {};

\begin{scope}[xshift=10cm]

\draw [rounded corners] (12, 1) -- (12, 3) -- (9, 3) -- (9, 1);

\draw [rounded corners] (11.5, 1) -- (11.5, 2.5) -- (9.5, 2.5) --
(9.5, 1);

\draw [rounded corners] (11, 1) -- (11, 2) -- (10, 2) -- (10, 1);

\draw [rounded corners] (12, 1) -- (12, -1) -- (7, -1) -- (7, 0);

\draw [rounded corners] (11.5, 0) -- (11.5, -0.5) -- (7.5, -0.5) --
(7.5, 0);

\draw [thick, red] (6, 0) node [black, below right] {$I'_-$} -- (11.5,
0) -- (11.5, 1) -- (6, 1) node [black, above right] {$I'_+$} -- cycle;

\draw (6, 0) node [red, point, label=above right:${t_0}$] {};

\draw (11.5, 0) node [red, point, label=above left:${t'_1}$] {};

\end{scope}

\end{tikzpicture}
\end{center}
\caption{Rauzy induction.} \label{fig:rauzy}
\end{figure}
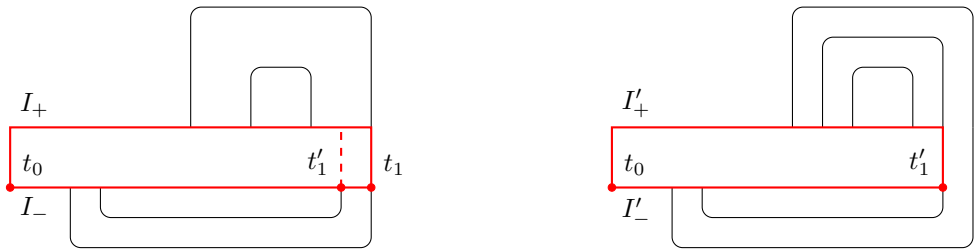

This process is called \emph{Rauzy induction} and is
a special case of splitting a train track.  For general non-classical
interval exchanges, Rauzy induction need not always be defined, but
the following result shows that it is always defined for the invariant
foliations of a pseudo-Anosov map.

\begin{proposition} \label{prop:rauzy} %
Let $g$ be a pseudo-Anosov map on a surface $S$, let $q$ be a choice
of quadratic differential corresponding to the invariant Teichm\"uller
geodesic, and let $S_q$ be the corresponding flat surface. Let
$t = t^1$ be an admissible horizontal transversal.  Then there is an
infinite sequence of horizontal transversals $( t^n )_{n \in \N}$, in
which $t^{n+1}$ is obtained from $t^n$ by Rauzy induction.
Furthermore, the length of the horizontal transversals $t^n$ tends to
zero as $n$ tends to infinity.
\end{proposition}

Replacing $g$ by $g^{-1}$ in Proposition \ref{prop:rauzy} switches the
horizontal and vertical directions, so the result also holds for a
vertical transversal for the horizontal foliation $F_-$.  Proposition
\ref{prop:rauzy} is an immediate consequence of the next two results,
Lemma \ref{lemma:saddle} and Proposition \ref{prop:bl}.

A \emph{saddle connection} for a flat surface is a geodesic connecting
two singular points, whose interior is disjoint from the singular set.
we say a saddle connection is vertical if it is contained in the
vertical foliation.  The following fact is well known, see for example 
\cite{farb-margalit}*{Lemma 14.11}.

\begin{lemma} \label{lemma:saddle} \cite{farb-margalit}*{Lemma 14.11} %
The invariant foliations for a pseudo-Anosov map do not contain any
saddle connections.
\end{lemma}

In this case, Rauzy induction is always defined, and gives an infinite
sequence of transversals $t_n$ determining non-classical interval
exchanges:

\begin{proposition} \cite{boissy-lanneau}*{Proposition
  4.2} \label{prop:bl} %
If the vertical foliation $F_+$ has no saddle connections, then Rauzy
induction starting at a horizontal transversal $t^1$ gives an infinite
sequence of transversals $(t_n)_{n \in \N}$ determining non-classical
interval exchanges, such that the horizontal length of the
transversals $t_n$ tends to zero as $n$ tends to infinity.
\end{proposition}

\subsection{Maximally filling curves}

Given a non-singular point $x$ in the flat surface $S_q$ there are
positive numbers $\delta$ and $\e$ such that $x$ is the bottom left
corner of an embedded Euclidean rectangle $B_0 = s \cross t$, disjoint
from the singular set, where $\norm{s} = \delta$ and $\norm{t} = \e$.
In particular, $s$ is a horizontal transversal for the vertical
foliation $F_+$ and $t$ is a vertical transversal for the horizontal
foliation $F_-$.  We may therefore use $B_0$ as the base rectangle for
a non-classical interval exchange in both the horizontal and vertical
directions.  Furthermore, possibly after passing to a subrectangle
$[0, \delta'] \cross [0, \e'] \subseteq [0, \delta] \cross [0, \e]$,
we may assume that both non-classical interval exchanges are
admissible.  By abuse of notation we shall relabel the new rectangle
as $B_0$, and relabel the sides as $s$ and $t$ of lengths $\delta$ and
$\e$.

For the vertical non-classical interval exchange, let $d_+$ be the
union of the singular flow lines for the vertical foliation $F_+$,
together with the flow lines incident to the initial vertical edge
$\{ 0 \} \cross [0, \e]$.  For the horizontal non-classical interval
exchange, let $d_-$ be the union of the singular flow lines for the
horizontal foliation $F_-$, together with the flow lines incident to
the initial horizontal edge $[0, \delta] \cross \{ 0 \}$.

This determines a cell decomposition of the surface as follows.

\begin{enumerate}

\item The vertices $V$ are the singular points of $S_q$, together with
all intersection points $d_+ \cap d_-$, and all endpoints of $d_+$ and
$d_-$ with the base rectangle $B_0$.

\item The edges $E$ are the connected components of
$( \partial B \cup d_+ \cup d_- ) \setminus V$.

\item The $2$-cells or faces $F$ are the connected components of
$S \setminus (\partial B \cup d_+ \cup d_-)$.

\end{enumerate}

The collection of faces consists of the base rectangle $B_0$, together
with a collection of discs, each of which is a Euclidean rectangle in
$S_q$ with two vertical edges, each of which is contained either in
$d_+$ or the vertical boundary of $B_0$, and two horizontal edges,
each of which is contained either in $d_-$ or the horizontal boundary
of $B_0$.

The vertices are one of the following types:

\begin{enumerate}

\item Trivalent vertices in the interiors of the sides of $B_0$.

\item $4$-valent vertices corresponding to the corners of $B_0$, and
all non-singular intersection points of $d_+ \cap d_-$.

\item $1$-valent and $6$-valent vertices corresponding respectively to
the $1$-prong and trivalent singular points of $S_q$.

\end{enumerate}

\begin{proposition} \label{prop:twice} %
Let $v$ be a simple closed curve carried by the vertical train track,
which passes over every vertical band at least twice.  Let $w$ be a
horizontal curve carried by the horizontal train track, which passes
over every horizontal band at least twice.  Then $v$ and $w$ are
maximally filling.
\end{proposition}

\begin{proof}
We may isotope $v$ to consist of a union of vertical flow lines,
together with geodesic arcs in $B_0$, each of which has endpoints in
the opposite horizontal sides of $B_0$.  Similarly, we may isotope $w$
to consist of a union of horizontal flow lines, together with geodesic
arcs in $B_0$, each of which has endpoints in the opposite vertical
sides of $B_0$.

The intersections of $v$ and $w$ with each face of the cell structure
for $S_q$ is:

\begin{figure}[H]
\begin{center}
\begin{tikzpicture}

\tikzstyle{point}=[circle, draw, fill=black, inner sep=1pt]

\draw [thick, red] (0, 0) -- (4, 0);
\draw [thick, red] (0, 3) -- (4, 3);

\draw [thick] (0, 0) -- (0, 3);
\draw [thick] (4, 0) -- (4, 3);

\draw (2, -0.5) node {$B_0$};
\draw (0, 0) node [point] {};
\draw (0, 3) node [point] {};
\draw (4, 0) node [point] {};
\draw (4, 3) node [point] {};

\draw (1, 0) node [point] {};
\draw (2.5, 0) node [point] {};

\draw (1.5, 3) node [point] {};
\draw (3, 3) node [point] {};

\draw (0, 0.75) node [point] {};
\draw (0, 1.75) node [point] {};

\draw (4, 1.25) node [point] {};
\draw (4, 2.25) node [point] {};

\draw [thick, blue] (0.25, 0) -- (0.5, 3);
\draw [thick, blue] (0.65, 0) -- (1, 3);
\draw [thick, blue] (1.25, 0) -- (2, 3);
\draw [thick, blue] (1.9, 0) -- (2.5, 3);
\draw [thick, blue] (3.15, 0) -- (3.5, 3);
\draw [thick, blue] (3.5, 0) -- (3.75, 3);

\draw [thick, green] (0, 0.25) -- (4, 0.5);
\draw [thick, green] (0, 0.5) -- (4, 1);
\draw [thick, green] (0, 1) -- (4, 1.5);
\draw [thick, green] (0, 1.5) -- (4, 2);
\draw [thick, green] (0, 2) -- (4, 2.5);
\draw [thick, green] (0, 2.5) -- (4, 2.75);

\begin{scope}[xshift=6cm]

\draw [thick, red] (0, 0) -- (3, 0);
\draw [thick, red] (0, 3) -- (3, 3);

\draw [thick] (0, 0) -- (0, 3);
\draw [thick] (3, 0) -- (3, 3);

\draw (0, 0) node [point] {};
\draw (0, 3) node [point] {};
\draw (3, 0) node [point] {};
\draw (3, 3) node [point] {};

\draw [thick, blue] (1, 0) -- (1, 3);
\draw [thick, blue] (2, 0) -- (2, 3);

\draw [thick, green] (0, 1) -- (3, 1);
\draw [thick, green] (0, 2) -- (3, 2);

\end{scope}

\end{tikzpicture}
\end{center}
\caption{The curves $v$ and $w$.} \label{fig:faces}
\end{figure}
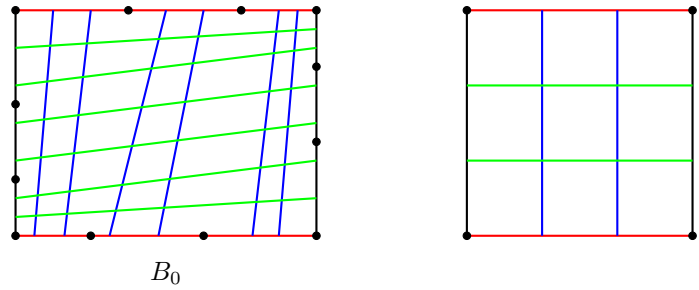

If a complementary region of $v \cup w$ is contained in the interior
of a single face, then it is a square.

If it intersects an edge of the cell structure, but is disjoint from
the vertices, then it is a square.

\begin{figure}[H]
\begin{center}
\begin{tikzpicture}

\tikzstyle{point}=[circle, draw, fill=black, inner sep=1pt]

\draw [thick, red] (0, 0) -- (3, 0);
\draw (0, 0) node [point] {};
\draw (3, 0) node [point] {};

\draw [thick, blue] (1, -1) -- (1, 1);
\draw [thick, blue] (2, -1) -- (2, 1);

\draw [thick, green] (1, -1) -- (2, -1);
\draw [thick, green] (1, 1) -- (2, 1);

\begin{scope}[xshift=6cm, yshift=-1.5cm]

\draw [thick] (0, 0) -- (0, 3);
\draw (0, 0) node [point] {};
\draw (0, 3) node [point] {};

\draw [thick, blue] (-1, 1) -- (-1, 2);
\draw [thick, blue] (1, 1) -- (1, 2);

\draw [thick, green] (-1, 1) -- (1, 1);
\draw [thick, green] (-1, 2) -- (1, 2);

\end{scope}

\end{tikzpicture}
\end{center}
\caption{A complementary region intersecting an
  edge.} \label{fig:square}
\end{figure}
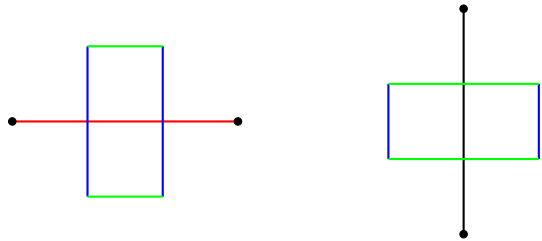

If it contains a vertex which is either trivalent or $4$-valent, then
the component is square.

\begin{figure}[H]
\begin{center}
\begin{tikzpicture}

\tikzstyle{point}=[circle, draw, fill=black, inner sep=1pt]

\draw [thick, red] (-1.5, 0) -- (1.5, 0);
\draw [thick] (0, -1.5) -- (0, 0);
\draw (0, 0) node [point] {};

\draw [thick, blue] (-1, -1) -- (-1, 1);
\draw [thick, blue] (1, -1) -- (1, 1);

\draw [thick, green] (-1, -1) -- (1, -1);
\draw [thick, green] (-1, 1) -- (1, 1);

\begin{scope}[xshift=5cm]

\draw [thick, red] (-1.5, 0) -- (0, 0);
\draw [thick] (0, -1.5) -- (0, 1.5);
\draw (0, 0) node [point] {};

\draw [thick, blue] (-1, -1) -- (-1, 1);
\draw [thick, blue] (1, -1) -- (1, 1);

\draw [thick, green] (-1, -1) -- (1, -1);
\draw [thick, green] (-1, 1) -- (1, 1);

\end{scope}

\begin{scope}[xshift=10cm]

\draw [thick, red] (-1.5, 0) -- (1.5, 0);
\draw [thick] (0, -1.5) -- (0, 1.5);
\draw (0, 0) node [point] {};

\draw [thick, blue] (-1, -1) -- (-1, 1);
\draw [thick, blue] (1, -1) -- (1, 1);

\draw [thick, green] (-1, -1) -- (1, -1);
\draw [thick, green] (-1, 1) -- (1, 1);

\end{scope}

\end{tikzpicture}
\end{center}
\caption{Complementary regions containing vertices.} \label{fig:vertex}
\end{figure}
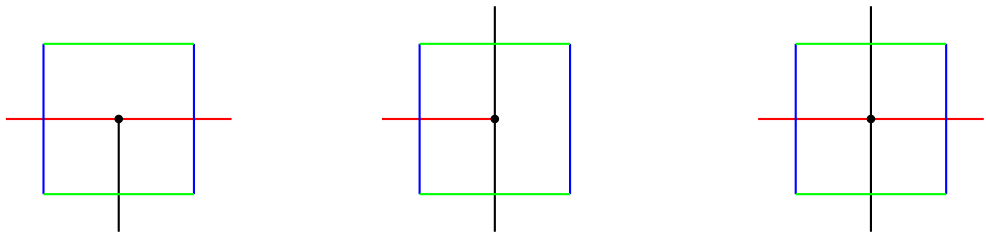

Finally, if it contains a singular point of $S_q$, then this singular
point is either a $1$-prong singularity, or a trivalent singular
point.  In the first case the component is a bigon containing a
puncture, and in the second case the component is a hexagon.  This is
illustrated in Figure \ref{fig:hexagon}.

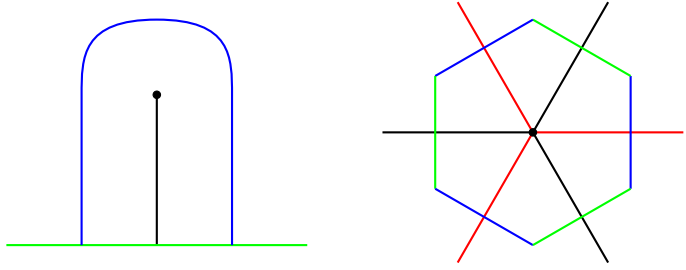
\begin{figure}[H]
\begin{center}
\begin{tikzpicture}

\tikzstyle{point}=[circle, draw, fill=black, inner sep=1pt]

\draw [thick, red] (0, 0) -- (0:2);
\draw [thick] (0, 0) -- (60:2);
\draw [thick, red] (0, 0) -- (120:2);
\draw [thick] (0, 0) -- (180:2);
\draw [thick, red] (0, 0) -- (240:2);
\draw [thick] (0, 0) -- (300:2);
\draw (0, 0) node [point] {};

\draw [thick, blue] (-30:1.5) -- (30:1.5);

\draw [thick, green] (30:1.5) -- (90:1.5);

\draw [thick, blue] (90:1.5) -- (150:1.5);

\draw [thick, green] (150:1.5) -- (210:1.5);

\draw [thick, blue] (210:1.5) -- (270:1.5);

\draw [thick, green] (270:1.5) -- (-30:1.5);

\begin{scope}[xshift=-5cm, yshift=-1.5cm]

\draw (0, 2) node [point] {};
\draw [thick] (0, 0) -- (0, 2);
\draw [thick, green] (-2, 0) -- (2, 0);
\draw [thick, blue]
(-1, 0) --
(-1, 2) .. controls (-1, 2.5) and (-1, 3) ..
(0, 3) .. controls (1, 3) and (1, 2.5) ..
(1, 2) --
(1, 0);

\end{scope}

\end{tikzpicture}
\end{center}
\caption{Complementary regions containing singularities of
  $S_q$.} \label{fig:hexagon}
\end{figure}

\end{proof}

\subsection{Widely separated curves are maximally filling}

It is well known that if the vertical foliation is uniquely ergodic,
then the intersection of the projective measured foliations carried by
the train tracks arising from Rauzy induction is equal to the vertical
foliation.  We provide a proof for the convenience of the reader.

\begin{proposition} \label{prop:intersection} %
Let $S_q$ be a flat surface with uniquely ergodic vertical foliation
$F_+$.  Let $t^n$ be a sequence of nested horizontal transversals in
$S_q$ such that the length of $t_n$ tends to zero as $n$ tends to
infinity.  Let $\tau_n$ be the train track determined by $t^n$.  Then
$\bigcap \overline{P(\tau_n)} = \{ \overline{F_+} \}$ in $\PMF(S)$.
\end{proposition}

\begin{proof}
Let $\tau_n$ be the train track corresponding to $t^n$.  As the
transversals are nested, $t^{n+1} \subset t^n$, each train track is
carried by the previous one, $\tau_{n+1} \prec \tau_n$, and so
$P(\tau_{n+1}) \subseteq P(\tau_n)$.  Every train track $\tau_n$ in
the splitting sequence carries the vertical foliation $F_+$, so
$F_+ \in \bigcap P(\tau_n)$.  We now show
$\bigcap \overline{ P(\tau_n) } = \{ \overline{ F_+ } \}$.

Let $F_+(a)$ be the measure of $a$ with respect to the vertical
measured foliation $F_+$, and let $F_-(a)$ be the measure of $a$ with
respect to the horizontal foliation $F_-$.  Consider the ratio between
these two measure, $\rho_q(a) = F_-(a) / F_+(a)$, which is
projectively invariant, i.e.
$\rho_q(\lambda a) = F_-( \lambda a) / F_+(\lambda a) = \lambda F_-(a)
/ \lambda F_+(a) = \rho_q(a)$, and so defines a function on $\PMF(S)$.

If $a$ is carried by $\tau_n$, then it has a piecewise geodesic
representative in $S_q$ consisting of flow lines in bands, which are
vertical geodesics segments, and geodesic segments in the base
rectangle $B_0$, each of which has one endpoint in the lower
horizontal side, and one endpoint in the upper horizontal side.  Flow
lines have zero vertical measure in $F_+$, and positive horizontal
measure in $F_-$. Each segment in $B_0$ contributes at least $2 \e$ to
the horizontal measure of $a$, and at most $F_+(t_n)$ to the vertical
measure of $a$.  Therefore $\rho_q(a) \ge 2 \e / F_+(t_n)$, where
$F_+(t_n)$ is equal to the length of the horizontal geodesic $t_n$ in
the flat metric $S_q$.  As the length of $t_n$ tends to zero, any
fixed curve $a$ can only be carried by finitely many $\tau_n$.  In
particular, for any simple closed curve $a$, there is an $n$ such that
$a \not \in P(\tau_n)$.

This argument works for elements $\overline{F} \in \PMF(S)$.  Suppose
that $( \overline{a}_n )_{n \in \N}$ is a sequence of simple closed
curves converging to $\overline{F}$.  Then for any representative $F$
in $\mathcal{MF}(S)$ for $\overline{F}$, there is a sequence of
numbers $\lambda_n$ such that $\lambda_n a_n \to F$.  Then
$F_+(\lambda_n a_n) \to F_+(F)$, where $F_+(F)$ is equal to the
intersection number $i(F_+, F)$ of the two measured foliations.  In
particular, if $i(F_+, F) > 0$, then there is an $n$ such that
$F \not \in P(\tau_n)$.  Therefore all foliations in
$\bigcap P(\tau_n)$ have zero intersection number with $F_+$, but as
$F_+$ is uniquely ergodic, the set of all projective measured
foliations with zero intersection number with $F_+$ consists just of
$\{ \overline{ F_+ } \}$.
\end{proof}

We will also use the following results of Masur and Minsky and
Klarreich.  The first says that the vertex sequences arising from
splitting sequences of train tracks are unparameterized quasigeodesics
in the curve complex.

\begin{theorem} \cite{masur-minsky}*{Theorem
  1.3} \label{theorem:splits} %
Given a surface $S$, there are constants $Q$ and $c$ such that the
vertices of a nested train-track splitting sequence form an
unparameterized $(Q, c)$-quasigeodesic in $\C(S)$.
\end{theorem}

The next says that a $1$-neighbourhood of a simple closed curve which
runs over every branch of a train track is in fact carried by the
train track.

\begin{lemma} \cite{masur-minsky}*{Lemma 3.4} \label{lemma:nest} %
If $\tau$ is a maximal birecurrent train track and $a$ is a simple
closed curve which runs over every branch of $\tau$, then $N_1(a)
\subseteq P(\tau)$.
\end{lemma}

Finally, the boundary of the curve complex is the space of minimal
foliations.

\begin{theorem} \cite{klarreich}*{Theorem
  1.3} \label{theorem:boundary} %
The boundary at infinity of the curve complex $\C(S)$ is the space of
minimal foliations on $S$.
\end{theorem}

We may now prove Lemma \ref{lemma:maximal}.

\begin{proof}
Choose a non-singular point in $S_q$ and construct a base rectangle
$B_0 = t^1 \cross s^1$ disjoint from the singular set, such that both
vertical and horizontal non-classical interval exchanges are
admissible.  Let $(t^n)_{n \in \N}$ and $(s^n)_{n \in \N}$ be the
train tracks arising from applying Rauzy induction to $t^1$ and $s^1$
respectively.

After applying Rauzy induction to $t^1$, there is a transversal
$t^{k_1}$ such that each band runs at least twice over every band of
the original non-classical interval exchange.  Similarly, after
applying Rauzy induction to $s^1$, there is a transversal $s^{k_2}$
and such that each band runs at least twice over every band of the
original non-classical interval exchange.  To simplify notation,
choose $k = \max \{ k_1, k_2\}$, and consider the train tracks
$\tau_k$ determined by $t^k$, and $\sigma_k$ determined by $s^k$.

By Proposition \ref{prop:intersection}, Theorem \ref{theorem:splits}
and Theorem \ref{theorem:boundary}, the vertex sequence
$( V(\tau_n) )_{n \in \N}$ is a $(Q, c)$-quasigeodesic, whose limit
point in $\partial \C(S)$ is equal to $\overline{F_+}$.  The positive
limit point of the $(1, K_1)$-quasigeodesic $\alpha_g$ is also equal
to $\overline{F_+}$, so there is a constant $L$ such that all of the
$V(\tau_n)$ are contained in an $L$-neighbourhood of any
$(1, K_1)$-quasiaxis for $g$.

Recall that the Rauzy induction moves may be realised in $S_q$ by
extending the flow lines giving the boundaries of the bands.  As $F_+$
is minimal, every leaf is dense, and so for every $\tau_k$ there is a
$k' > k$ such that every band of $\tau_{k'}$ runs at least once over
every band of $\tau_k$.  The nesting lemma, Lemma \ref{lemma:nest},
then implies that $N_1(\C(\tau_{k'})) \subseteq \C(\tau_k)$.  By
repeating this process, for any $K$ there is a $k'$ such that
$N_K( \C(\tau_{k'}) ) \subseteq \C(\tau_k)$.  In particular, for any
$K + L$ there is a $k'$ such that
$N_{K+L}( \C(\tau_{k'}) ) \subset \C(\tau_k)$ in $\C(S)$.  Let
$\alpha_+ = \alpha_g( [ r_+, \infty ) )$ be the terminal ray of
$\alpha$ such that
$\alpha_+ \subseteq N_L( ( V(\tau_n) )_{n \ge k'} )$.  Then
$N_K(\alpha_+) \subseteq N_{K+L}( \C(\tau_{k'}) ) \subseteq
\C(\tau_k)$.

We may apply exactly the same argument to the horizontal foliation, so
there is an infinite initial ray
$\alpha_- = \alpha_g( ( -\infty, r_- ] )$ the $(1, K_1)$-quasiaxis
$\alpha_g$ such that $N_K(\alpha_-) \subseteq \C(\sigma_k)$.

As $g$ acts by translations on $\alpha_g$, given any two points $v$
and $w$ on $\alpha_g$ distance at least
$R = d_\C(\alpha_-, \alpha_+ ) + d_\C(a, g a) + O(\delta)$ apart, we
may translate them by powers of $g$ such that, up to relabelling, $v$
lies in $\alpha_-$ and $w$ lies in $\alpha_+$.  Let $v'$ and $w'$ be
curves distance at most $K$ from $v$ and $w$ respectively.  Then
$v' \in \C(\sigma_k)$ and $w' \in \C(\tau_k)$, so $v'$ and $w'$ are
maximally filling, by Proposition \ref{prop:twice}.  Therefore, $v$
and $w$ are $K$-maximally filling, as required.
\end{proof}

\begin{lemma}\label{lem.max-bdd-diam}
Let $g$ be a pseudo-Anosov element with maximal invariant laminations
and axis $\alpha$.  Then there is a $D$ such that for any non-maximal
train track $\tau$, the nearest point projection of $\C(\tau)$ to
$\alpha$ has diameter at most $D$.
\end{lemma}

\begin{proof}
The set $\C(\tau)$ is uniformly quasiconvex, so if there is a
projection of size $R$, then there are curves $a_1, a_2$ in $\C(\tau)$
with $d_{\C(S)}(\alpha, a_i) \le K$ and
$d_{\C(S)}(a_1, a_2) \ge R - 2K - O(\delta)$, where $K$ is a uniform
constant independent of $\C(\tau)$.

As $a_1$ and $a_2$ are both carried by $\tau$, they have a
complementary region which is not either a triangle or a monogon with
exactly one puncture, so this contradicts the lemma above.
\end{proof}

Now by Lemma \ref{lem.max-bdd-diam} combined with Theorem
\ref{thm.WPD}, we have the following theorem which corresponds to the
statement for $\C_{\tau}(S)$ in Theorem \ref{thm.main}.
\begin{theorem}
Let $S$ be an orientable surface of genus $g$ with $n$ punctures.
Suppose $2g + n > 3$.
For the action of $\mathrm{MCG}(S)$ on $\C_{\tau}(S)$, we have
\begin{enumerate}
\item[(a)]  
$g\in\mathrm{MCG}(S)$ is loxodromic if and only if $g$ has maximal invariant laminations, and
\item[(b)] the action of $\mathrm{MCG}(S)$ on $\C_{\tau}(S)$ is strongly WPD.
\end{enumerate}
In particular, the action of $\mathrm{MCG}(S)$ on $\C_{\tau}(S)$ is non-elementary.
\end{theorem}

Lemma \ref{lem.max-bdd-diam} holds for all surfaces with $3g + n > 4$.
However, part $(a)$ of this result also holds for $S_{0, 4}$ by Remark
\ref{remark:stratum}.  In fact, $(b)$ also holds, as the non-maximal
train tracks in $S_{0, 4}$ have finite diameter in $\C(S_{0, 4})$, and
so the result holds for all surfaces with $2g + n > 3$.  Among
surfaces suppporting pseudo-Anosov elements this excludes only
$S_{1, 1}$, for which the result does not hold.





\vskip 20pt

\noindent Joseph Maher \\
CUNY College of Staten Island and CUNY Graduate Center \\
\url{joseph.maher@csi.cuny.edu} \\

\noindent Hidetoshi Masai \\
Tokyo Institute of Technology\\
\url{masai@math.titech.ac.jp}\\

\noindent Saul Schleimer \\
University of Warwick \\
\url{s.schleimer@warwick.ac.uk }


\end{document}